\documentclass[12pt]{article}
\usepackage{amsmath}
\usepackage{amssymb}
\usepackage{theorem}

    \usepackage{hyperref}

\newtheorem{thm}{Theorem}[section]
\newtheorem{prop}[thm]{Proposition}

\newcommand{\qed}{\hfill \mbox{\raggedright \rule{.07in}{.1in}}}
 
\newenvironment{proof}{\vspace{1ex}\noindent{\bf
Proof}\hspace{0.5em}}{\hfill\qed\vspace{1ex}}
\newenvironment{pfof}[1]{\vspace{1ex}\noindent{\bf Proof of
#1}\hspace{0.5em}}{\hfill\qed\vspace{1ex}}

\newcommand{\E}{{\mathbb E}}
\newcommand{\cB}{{\mathcal B}}
\newcommand{\cG}{{\mathcal G}}

\newcommand{\SMALL}{\textstyle}

\title{Nonexistence of spectral gaps in H\"older spaces for continuous time dynamical systems}

\author{Ian Melbourne\thanks{Mathematics Institute, University of Warwick, Coventry, CV4 7AL, UK}
\and Nicol\`o Paviato
\thanks{Mathematics Institute, University of Warwick, Coventry, CV4 7AL, UK}
 \and Dalia Terhesiu\thanks{Mathematisch Instituut,
University of Leiden, Niels Bohrweg 1, 2333 CA Leiden, Netherlands}}

\date{27 February 2021}

\begin{document}

 \maketitle

\begin{abstract}
We show that there is a natural restriction on the smoothness of spaces where the transfer operator for a continuous dynamical system has a spectral gap.
Such a space cannot be embedded in a H\"older space with
H\"older exponent greater than $\frac12$ unless it consists entirely of coboundaries.
\end{abstract}

 \section{Introduction} 
 \label{sec-intro}

Decay of correlations (rates of mixing) and strong statistical properties are well-understood for Axiom~A diffeomorphisms since the work of~\cite{Bowen75,Ruelle78,Sinai72}.
Mixing rates are computed with respect to any equilibrium measure with H\"older potential.
Up to a finite cycle, such diffeomorphisms have exponential decay of correlations for H\"older observables.   In the one-sided (uniformly expanding) setting, this is typically proved by establishing quasicompactness and a spectral gap for the associated transfer operator $L$.  Such a spectral gap yields a decay rate $\|L^n v-\int v\|\le C_v e^{-an}$ for $v$ H\"older, where $\|\;\|$ is a suitable H\"older norm and $a$, $C_v$ are positive constants.  Decay of correlations for H\"older observables is an immediate consequence of the decay for $L^n$.
This philosophy has been extended to large classes of nonuniformly expanding dynamical systems with exponential~\cite{Young98} and subexponential decay of correlations~\cite{Young99}.

For continuous time dynamical systems, the usual techniques~\cite{Dolgopyat98a,Liverani04,Pollicott85} bypass spectral gaps; the only exceptions that we know of being Tsujii~\cite{Tsujii08,Tsujii10}.  However, the result in~\cite{Tsujii08} is for suspension semiflows over the doubling map with a $C^3$ roof function, where the smoothness of the roof function is crucial and very restrictive.  A spectral gap for contact Anosov flows is obtained in~\cite{Tsujii10}; unfortunately it seems nontrivial to extend this to nonuniformly hyperbolic contact flows (or uniformly hyperbolic contact flows with unbounded distortion), see~\cite{BaladiDemersLiverani18} which proves exponential decay of correlations for billiard flows with a contact structure but does not establish a spectral gap.  Indeed, apart from~\cite{Tsujii08,Tsujii10}, there are no results on spectral gaps of transfer operators for semiflows and flows.

The results of Tsujii~\cite{Tsujii08,Tsujii10} provide a spectral gap in an anisotropic Banach space.  In this paper we obtain a restriction on the Banach spaces that can yield a spectral gap.  We work in the following general setting:

Let $(\Lambda,d)$ be a bounded metric space with Borel probability measure $\mu$, and
let $F_t:\Lambda\to\Lambda$ be a measure-preserving semiflow.
We suppose that $t\to F_t$ is Lipschitz a.e.\ on $\Lambda$.
Let $L_t:L^1(\Lambda)\to L^1(\Lambda)$ denote the transfer operator corresponding to $F_t$ (so $\int_\Lambda L_tv\,w\,d\mu=\int_\Lambda v\,w\circ F_t\,d\mu$ for all $v\in L^1(\Lambda)$, $w\in L^\infty(\Lambda)$, $t>0$).
Let $v\in L^\infty(\Lambda)$ and define $v_t=\int_0^t v\circ F_r\,dr$ for $t\ge0$.

\begin{thm}  \label{thm:anti}
Let $\eta\in(\frac12,1)$.
Suppose that $L_tv\in C^\eta(\Lambda)$ for all $t>0$ and that
$\int_0^\infty \|L_tv\|_\eta\,dt<\infty$.
Then $v_t$ is a coboundary:
\[
v_t=\chi\circ F_t-\chi \quad\text{for all $t\ge0$, a.e.\ on $\Lambda$}
\]
where $\chi=\int_0^\infty L_tv\,dt\in C^\eta(\Lambda)$.
In particular, $\sup_{t\ge0}|v_t|_\infty<\infty$.
\end{thm}
Here, $|g|_\infty=\operatorname{ess\, sup}_\Lambda |g|$ and $\|g\|_\eta=|g|_\infty+\sup_{x\neq y}|g(x)-g(y)|/d(x,y)^\eta$.

Theorem~\ref{thm:anti} implies that any Banach space admitting a spectral gap and embedded in
$C^\eta(\Lambda)$ for some $\eta>\frac12$ is cohomologically trivial.
However, for (non)uniformly expanding semiflows and (non)uniformly hyperbolic flows of the type in the aforementioned references,
coboundaries are known to be exceedingly rare, see for example~\cite[Section~2.3.3]{CFKMZ19}.  Hence, Theorem~\ref{thm:anti} can be viewed as an ``anti-spectral gap'' result for such continuous time dynamical systems.

\section{Proof of Theorem~\ref{thm:anti}}
\label{sec:anti}

Let $v\in L^\infty(\Lambda)$, with $L_tv\in C^\eta(\Lambda)$ for all $t>0$ and
$\int_0^\infty \|L_tv\|_\eta\,dt<\infty$ where $\eta\in(\frac12,1)$.
Following Gordin~\cite{Gordin69} we consider a martingale-coboundary decomposition.
Define $\chi=\int_0^\infty L_tv\,dt\in C^\eta(\Lambda)$, and
\[
v_t=\int_0^t v\circ F_r\,dr, \qquad
m_t=v_t-\chi\circ F_t+\chi,
\]
for $t\ge0$.
Let $\cB$ denote the Borel $\sigma$-algebra on $\Lambda$.

\begin{prop} \label{prop:anti}
\noindent (i) $t\to m_t$ is $C^\eta$ a.e.\ on $\Lambda$.

\vspace{1ex}
\noindent (ii) $\E(m_t|F_t^{-1}\cB)=0$ for all $t\ge0$.
\end{prop}

\begin{proof}
(i) For $0\le s\le t\le 1$ and $x\in\Lambda$,
\begin{align*}
|m_s(x)-m_t(x)| & \le |v_s(x)-v_t(x)|+|\chi(F_sx)-\chi(F_tx)|
\\ & \le |s-t||v|_\infty+|\chi|_\eta\, d(F_sx,F_tx)^\eta.
\end{align*}
Since $t\mapsto F_t$ is a.e.\ Lipschitz, it follows that
$t\mapsto m_t$ is a.e.\ $C^\eta$.

\vspace{1ex}
\noindent (ii) Let $U_tv=v\circ F_t$, and recall that $L_tU_t=I$ and
$\E(\cdot|F_t^{-1}\cB)=U_tL_t$.  Then
\begin{align*}
L_tm_t & = \SMALL  L_t(v_t-U_t\chi+\chi)
=\int_0^t L_tU_rv\,dr - \chi + \int_0^\infty L_tL_r v\,dr 
\\[.75ex] &  \SMALL =\int_0^t L_{t-r}v\,dr - \chi + \int_0^\infty L_{t+r} v\,dr 
=\int_0^t L_rv\,dr - \chi + \int_t^\infty L_r v\,dr =0.
\end{align*}
Hence $\E(m_t|F_t^{-1}\cB)=U_tL_tm_t=0$.
\end{proof}

\begin{pfof}{Theorem~\ref{thm:anti}}
Fix $T>0$, and define 
\[
M_T(t)=m_T-m_{T-t}=m_t\circ F_{T-t},\quad t\in[0,T].
\]
Also, define the filtration $\cG_{T,t}=F_{T-t}^{-1}\cB$.
It is immediate that
$M_T(t)= m_t\circ F_{T-t}$ is $\cG_{T,t}$-measurable.
Also, for $s<t$ we have $M_T(t)-M_T(s)=m_{T-s}-m_{T-t}=m_{t-s}\circ F_{T-t}$,
so
\[
\E(M_T(t)-M_T(s)|\cG_{T,s})=
\E(m_{t-s}\circ F_{T-t}|F_{T-s}^{-1}\cB)
=\E(m_{t-s}|F_{t-s}^{-1}\cB)\circ F_{T-t}=0
\]
by Proposition~\ref{prop:anti}(ii). 
Hence $M_T$ is a martingale for each $T>0$.
Next, 
\[
|M_T(t)|_\infty
=|m_t\circ F_{T-t}|_\infty
\le |m_t|_\infty \le |v_t|_\infty+2|\chi|_\infty
\le T|v|_\infty+2|\chi|_\infty.
\] 
Hence $M_T(t),\,t\in[0,T]$, is a bounded martingale.

By Proposition~\ref{prop:anti}(i), $M_T$ has $C^\eta$ sample paths.
Since $\eta>\frac12$, it follows from general martingale theory that $M_T\equiv 0$ a.e.  
Taking $t=T$, we obtain $m_T=0$ a.e.
Hence $v_T=\chi\circ F_T-\chi$ a.e.\ for all $T>0$ as required.

For completeness, we include the argument that $M_T\equiv0$ a.e.
We require two standard properties of the quadratic variation process $t\mapsto [M_T](t)$; a reference for these is~\cite[Theorem~4.1]{ChungWilliams}.  First, $[M_T](t)$ is the limit in probability 
as $n\to\infty$ of 
\[
S_n(t)=\sum_{j=1}^n(M_T(jt/n)-M_T((j-1)t/n))^2.
\]
Second (noting that $M_T(0)=0$),
\[
[M_T](t)=M_T(t)^2-2\int_0^t M_T\,dM_T,
\]
where the stochastic integral has expectation zero.
In particular, $\E([M_T])\equiv \E(M_T^2)$.

 Since $M_T$ has H\"older sample paths with exponent $\eta>\frac12$, we have a.e.\ that
\[
|S_n(t)|=O(t^\eta n^{-(2\eta-1)})\to0 \quad\text{as $n\to\infty$.}
\]
Hence $[M_T]\equiv0$ a.e.  It follows that
$\E(M_T^2)\equiv0$ and so $M_T\equiv0$ a.e.
\end{pfof}

\end{document}